\documentclass{amsart}
\usepackage{amsfonts}
\usepackage{amsmath}
\usepackage{amssymb}
\usepackage{amsthm}
\usepackage[usenames]{color}

\numberwithin{equation}{section}

\newcommand{\Z}{\mathbb{Z}}
\newcommand{\N}{\mathbb{N}}

\newtheorem{theorem}{Theorem}[section]

\newtheorem{lemma}[theorem]{Lemma}
\newtheorem{corollary}[theorem]{Corollary}

\theoremstyle{definition}

\begin{document}

\title[Weighted EGZ Constant]{Weighted EGZ Constant for $p$-groups of rank 2}

\author{Filipe A. A. Oliveira, Hemar T. Godinho \and Ab\'{i}lio Lemos}

\address{Instituto Federal de Educa\c c\~ao, Ci\^encia e Tecnologia do Rio Grande do Sul-Campus Farroupilha, Farroupilha-RS, Brazil}

\email{filipe.oliveira@farroupilha.ifrs.edu.br}

\address{Departamento de Matem\'{a}tica, Universidade Federal de Vi\c cosa, Vi\c cosa-MG, Brazil}

\email{abiliolemos@ufv.com.br}

\address{Departamento de Matem\'{a}tica, Universidade de Bras\'ilia, Bras\'ilia-DF, Brazil}

\email{hemar@mat.unb.br}



\begin{abstract}
Let $G$ be a finite abelian group of exponent $n$, written additively, and let $A$ be a subset of $\Z$. The constant $s_A(G)$ is defined as the smallest integer $\ell$ such that any sequence over $G$ of length at least $\ell$ has an $A$-weighted zero-sum of length $n$ and  $\eta_A(G)$ defined as the smallest integer $\ell$ such that any sequence over $G$ of length at least $\ell$ has an $A$-weighted zero-sum of length at most $n$. Here we prove that, for $\alpha \geq \beta$,  and $A=\left\{x\in\mathbb{N}\; : \; 1 \le a \le p^{\alpha} \; \mbox{ and }\; \gcd(a, p) = 1\right \}$, we have  $s_{A}(\Z_{p^{\alpha}}\oplus \Z_{p^\beta}) = \eta_A(\Z_{p^{\alpha}}\oplus \Z_{p^\beta}) + p^{\alpha}-1 = p^{\alpha} + \alpha +\beta$ and classify all the extremal  $A$-weighted zero-sum free sequences.

\end{abstract}

\maketitle

\section{Introduction}
 
Let $G$ be a finite abelian group of exponent $n$, written additively, and let $A$ be a subset of $\Z$. The constant $s_A(G)$ is defined as the smallest integer $\ell$ such that any sequence $x_1x_2\cdots x_m$ of elements of $G$ with $m\ge \ell$ has a subsequence $x_1^{\ast}\cdots x_n^{\ast}$ such that $a_1x_1^{\ast}+ \cdots+ a_nx_n^{\ast} = 0$ in $G$, where $a_1, \ldots, a_n \in A$. The set $A$ is called the set of weights, and the subsequence $x_1^{\ast} \cdots x_n^{\ast}$ is called an $A$-weighted zero-sum sequence of length $n$. It is habitual to pair up the constant $s_A(G)$ with the constant $\eta_A(G)$ defined as the smallest integer $\ell$ such that any sequence over $G$ of length at least $\ell$ has an $A$-weighted zero-sum of length at most $n$. The question here is to find lower and upper bounds (or better yet, exact values) for the constants $s_A(G)$ and $\eta_A(G)$.

When considering $A = \{1\}$, the constant $s_A(G)$ is known as the Erd\"os-Ginzburg-Ziv constant (or the EGZ constant), and is denoted by $s(G)$. This is a classical area of research and we refer the reader to \cite{gaogersurvey} and \cite{chintamietal} for a survey and recent contributions to the theory.

Let us now consider the set of weights
$$
A = \{a \in \N \, | \, 1 \le a \le n, \mbox{ and } \gcd(a, n) = 1\}.
$$
In the last years many authors have considered these constants $s_A(G)$ and $\eta_A(G)$, associated with the set $A$ above. We list here some of these contributions.
\begin{enumerate}
\item[(i)] $s_A(\Z_2^r)=2^r+1$ (see \cite{ha});
\item[(ii)] $s_A(\Z_3)=4$ (see \cite{Ad2}), $s_A(\Z_3^2)=5$ (see \cite{Ad3}), $s_A(\Z_3^3)=9$ (see \cite{GL,MO}), $s_A(\Z_3^4)=21$, $s_A(\Z_3^5)=41$ and $s_A(\Z_3^6)=113$ (see \cite{MO});
\item[(iii)] $s_A(\Z_4)=6$ and $s_A(\Z_6)=8$ (see \cite{Ad2}); $s_A(\Z^2_4)=8$ (see \cite{Ad});
\item[(iv)] $s_A(\Z_2\oplus \Z_4)=7$ (see \cite{MO2,BK}), $s_A(\Z_2^2\oplus \Z_4)=8$ (see \cite{BK}) and $s_A(\Z_2\oplus \Z_6)=9$ (see \cite{MO2});
\end{enumerate}

The following results relate these two constants.
\begin{enumerate}
\item[(i)] $s_A(\Z_2^r)=\eta_A(\Z_2^r)+2-1$ (see \cite{ha});
\item[(ii)] $s_A(\Z_3)=\eta_A(\Z_3)+3-1$ (see \cite{Ad2}), $s_A(\Z_3^2)=\eta_A(\Z_3^2)+3-1$ (see \cite{Ad3,BK});
\item[(iii)] $s_A(\Z_4)=\eta_A(\Z_4)+4-1$ and $s_A(\Z_6)=\eta_A(\Z_6)+6-1$ (see \cite{Ad2});
\item[(iv)] $s_A(\Z^2_4)=\eta_A(\Z_4^2)+4-1$ (see \cite{Ad,BK});
\item[(v)] $s_A(\Z_2\oplus \Z_4)=\eta_A(\Z_2\oplus \Z_4)+4-1$ (see \cite{MO2,BK}), $s_A(\Z_2^2\oplus \Z_4)=\eta_A(\Z_2^2\oplus \Z_4)+4-1$ (see \cite{BK}) and $s_A(\Z_2\oplus \Z_6)=\eta_A(\Z_2\oplus \Z_6)+6-1$ (see \cite{MO2});
\item[(vi)] $s_A(\Z_{p^s}\oplus \Z_{p^r})=\eta_A(\Z_{p^s}\oplus \Z_{p^r})+p^r-1$, where $p$ is an odd prime number and $s\in\left\{1,2\right\}$ (see \cite{CHI,CH2});
\end{enumerate}

The results above suggests that we should have $s_A(G) = \eta_A(G) + n - 1$, but this
was proved wrong by Godinho, Lemos and Marques (see \cite{GL}), who showed that
$s_A(\Z_3^r)=2\eta_A(\Z_3^r)-1 >\eta_A(\Z_3^r)+3-1$ for $r\geq3$.

In the case of $G = \Z_p^r$, for $p$ a prime number, Adhikari \textit{et al.} (see [2]) proved that $s_A(G) = p + r$, for all $p > r$. In this direction Luca [19] (see also [16]) proved that $s_A(\Z_n) = n+\Omega(n)$ and he classified the extremal $A$-weighted zero-sum free sequences for $n = p^k$, where $\Omega(n)$ denotes the total number of prime divisors of $n$ (counted with multiplicity). This result was conjectured by Adhikari \textit{et al.} (see [3]).

Recently, Chintamani and Paul (see \cite{CHI,CH2}) proved that $s_A(\Z_{p^s}\oplus \Z_{p^\alpha})= p^\alpha + \alpha +s$ and $\eta_A(\Z_{p^s}\oplus \Z_{p^\alpha})= \alpha +s + 1$, for $s\in\left\{1,2\right\}$, $\alpha\geq s$ and $p$  an odd prime number. They also classify the extremal $A$-weighted zero-sum free sequences and extended these results proving that $s_A(\Z_{p^s}\oplus \Z_{n})\leq n+\Omega(n)+2s$, for $s\in\left\{1,2\right\}$, provided $p^s|n$.

In this paper we present a generalization of these results, proving that
 $s_A(\Z_{p^{\alpha}}\oplus \Z_{p^\beta})=\eta_A(\Z_{p^{\alpha}}\oplus \Z_{p^\beta}) + p^\alpha -1 = p^\alpha + \alpha +\beta$, where $p$ is an odd prime number and $\alpha \ge \beta$ are positive integers. We also classify the extremal $A$-weighted zero-sum free sequences.
\section{Notations, terminologies and preliminary results}
 Let $\mathbb{N}_{0}= \mathbb{N} \cup \{ 0 \}$  and define $[a,b]=\left\{ x\in\mathbb{N}_{0}:a\leq x\leq b\right\} $ for $a,b\in \mathbb{N}_{0}$. Throughout this paper  we are going to consider $p$ an odd prime number, $\alpha,\beta \in \mathbb{N}$, 
 \begin{equation}\label{A}
 G=\Z_{p^{\alpha}}\oplus \Z_{p^\beta}\;\;  \mbox{and}\;\;  A=\left\{x\in\mathbb{N}\; : \; 1 \le a \le p^{\alpha} \; \mbox{ and }\; \gcd(a, p) = 1\right\}.
 \end{equation}
 
 Let  $S=x_1 x_2 \cdots x_m$ be a sequence of elements of $G$ and denote by $\left|S\right|=m$,  the length of $S$. If $T$ is a subsequence of $S$, we will represent it as $T|S$. If $S_{1}$ and $S_{2}$ are sequences over $G$, we represent by $S_{1}S_{2}$ the obvious sequence having $S_{1}$ and $S_{2}$ as subsequences. If $T|S$, we will represent the subsequence of $S$ obtained by extracting from $S$ all the terms of $T$ by $ST^{-1}$.
 
The proof of the next lemma is trivial and will be omitted.
\begin{lemma} \label{ob1}
Let  $S = x_1 x_2 \cdots  x_\ell$ be a sequence over $G$. Then $S$ is an $A$-weighted zero-sum sequence if, and only if, either  
$$
S' = (u_1x_1)(u_2x_2) \cdots (u_\ell x_\ell) \;\mbox{ or }\;
\Theta(S) =\Theta(x_1)\Theta(x_2) \cdots  \Theta(x_\ell)
$$
is  an $A$-weighted zero-sum sequence, with $u_1, \cdots, u_{\ell} \in A$ (see \eqref{A}), and  $\Theta \in \mbox{Aut}(G)$.
\end{lemma}
\begin{lemma} \label{h4}
Let $G=\Z_{p^{\alpha}}\oplus \Z_{p^\beta}$, with $\alpha \geq \beta$. Given an element $(a, b) \in G$ with $\gcd(a,p)=1$, there exists an automorphism $\Theta \in \mbox{Aut}(G)$ such that $\Theta(a,b) = (1,0)$ and $\Theta(0,1)=(0,1)$.
\end{lemma}
\begin{proof}

Let $\mu : \Z_{p^{\alpha}} \longrightarrow \Z_{p^{\beta}}$ be the canonical homomorphism defined as $\mu(u)\equiv u \pmod{p^{\beta}}$. Given $(a,b)\in G$ with $\gcd(a,p)=1$, there exists $a^{-1}\in \Z_{p^\alpha}$ such that $a\cdot a^{-1}=1$ in $\Z_{p^\alpha}$. Now it is simple to verify that 
$\Theta(x,y)=( xa^{-1} ,\, y - b\mu(a^{-1}x))$ is an automorphism of $G$,
and $\Theta(a,b) = (1,0)$,  $\Theta(0,1)=(0,1)$.
\end{proof}

\begin{lemma} \label{h1}
Let  $a_1, a_2, \ldots, a_m \in \Z$  and define $N(a_1, a_2, \ldots, a_m)$ to be the number of solutions of 
\begin{equation}\label{eq1}
a_1x_1 + a_2x_2 + \cdots + a_mx_m \equiv 0 \pmod p
\end{equation}
with $x_j \in [1, p-1]$, for all $j \in [1, m]$. If $\gcd(a_{1},p)=1$ then
$$N(a_1, a_2, \ldots, a_m)= (p-1)^{m-1} - N(a_2, \ldots, a_m).$$
\end{lemma}
\begin{proof}
Observe that $(x_1, x_2,\ldots, x_m)$ is a solution of \eqref{eq1} with 
$x_1,x_2,\ldots,x_m \in [1,p-1]$ if, and only if, $ a_2x_2 + \cdots + a_mx_m \not\equiv 0 \pmod{p}$, otherwise we would have 
$a_1x_1\equiv 0\pmod{p}$ and $\gcd(a_1x_1,p)=1$, an impossibility. On the other hand, for any $x_{2},\ldots, x_{m}\in [1,p-1]$ such that $a_2x_2 + \cdots + a_mx_m  = -b \not\equiv 0 \pmod{p}$, the equation $a_{1}x_{1} \equiv b \pmod{p}$ has exactly one solution. This concludes  the proof.
\end{proof}

We close this section with  a particular case of a result proved by F. Luca in \cite{LUC}, and an immediate consequence of this.

\begin{lemma}\label{luca} 
 Let  $r\geq 2$ and $S = a_1 a_2 \cdots a_{r}$, a sequence over $\Z_{p^{\alpha}}$. If there 
 are $i,j\in[1,r]$ such that $\gcd(a_{i}a_{j},p)=1$ then,  for any given $b\in \Z$, there exist
 $x_{1},\ldots, x_{r}\in A$  (see \eqref{A}) such that
$$a_{1}x_{1}+\cdots + a_{r}x_{r} \equiv b  \, \pmod{p^{\alpha}}.$$
\end{lemma}
\begin{lemma} \label{h21}
Let  $m \ge 2$ and  $a_1, a_2, \ldots, a_{m + 1}\in \Z$. Then there exists  a set of indexes $I\subset [1, m + 1]$ of length $|I| = m$ such that $\sum_{i \in I} u_ia_i \equiv 0 \pmod p$ for some choice of integers $u_i\in [1, p-1]$, for all $i \in I$.
\end{lemma}

\begin{proof}
The result is trivial if there are $m$ integers $a_{j}\equiv 0 \pmod{p}$,  otherwise it is direct consequence of Lemma \ref{luca}.
\end{proof}
\section{$A$-weighted zero-sum sequences}

\begin{lemma} \label{h2}
Let $p$ be an odd prime, $m \ge 5$ an integer, and $$S = (1,0) (a_2,b_2) (a_3, b_3) (a_4, b_4) \cdots (a_{m + 1}, b_{m + 1})$$ a sequence over $G$ (see \eqref{A})  where $\gcd(b_2,p)=1$, $p|a_2$. If there are two indexes $i, j \in [3, m + 1] $ (including the case $i = j$) such that $\gcd(a_ib_j, p) = 1$, then $S$ has an $A$-weighted zero-sum subsequence $S'$ of length $m$.
\end{lemma}

\begin{proof}
Let us suppose, without losing  generality, that either (i) $\gcd(a_{3}b_{3}, p) = 1$, or (ii) $\gcd(a_3b_4, p) = 1$, $p|a_4$ and $p|b_3$.
By Lemma \ref{h21}, making the necessary changes in the indexes, in any case we can suppose that there is a choice of integers $u_4, u_5, \ldots, u_m \in [1, p-1]$ such that $\sum_{i = 4}^m u_ia_i \equiv 0 \pmod p$. Hence $N(a_4,\ldots, a_{m}) \geq 1$, and $N(a_3,a_{4},\ldots, a_m) \leq (p-1)^{m-3}-1$, according to Lemma \ref{h1} for $\gcd(a_{3},p)=1$  in any case. On the other hand, also by Lemma  \ref{h1}, we have $N(b_3,\ldots, b_{m})\leq (p-1)^{m-3}$ for either  case (i), or case (ii).
Since  
$$(p-1)^{m-2} > 2(p-1)^{m-3} -1 \geq N(a_3,\ldots, a_{m}) + N(b_3,\ldots, b_{m}),$$
there must exist  $v_3, v_4, \ldots, v_m\in [1,p-1]$, such that
$$a_3v_3 + a_4v_4 + \cdots + a_mv_m = a, \;\;\;
b_3v_3 + b_4v_4 + \cdots + b_mv_m =b,
$$
and  $\gcd(ab, p) = 1$. Since $ \{-b_2a + ba_2, \; -b, \;b_2v_i \} \subset A$, for 
$\gcd(abb_2v_i,p)=1$ and $p|a_2$, we have 
$$(-b_2a+ba_2)(1,0)+(-b)(a_2,b_2) + \sum_{i=3}^m b_2v_i(a_i,b_i) = (0,0)\;\;\mbox{in $G$},$$
as wanted.
\end{proof}
\begin{lemma} \label{fi1}
Let   $G = \Z_{p^{\alpha}}\oplus \Z_{p^\beta}$,  $\alpha \ge \beta$ and $\ell \ge 4$.  Let $S = x_1x_2\cdots x_{\ell + \beta}$ be a sequence over $G$, with $x_i = (a_i, b_i)$, for  $i \in [1, \ell + \beta]$. If there are $i,j\in [1,\ell+\beta]$, $i\neq j$, such that $x_{i}=x_{j}$ and $\gcd(a_{i}a_{j},p)=1$,  then $S$ has an $A$-weighted zero-sum subsequence of length $\ell$.
\end{lemma}

\begin{proof} (Induction on $\beta$). The cases $\beta = 1\mbox{ or } 2$ were proved by Chintamani and Paul in \cite{CHI,CH2}.  We may then assume $\beta \ge 3$. 
Without loss of generality, take $x_1 = x_2$  with $\gcd(a_1a_2,p)=1$. It follows from Lemmas \ref{ob1} and \ref{h4}, that we can assume $x_1 = x_2 = (1, 0)$.

Suppose we can find $i_{1},i_{2} \in [3,\ell + \beta]$, $i_{1}\neq i_{2}$, such that $\gcd(b_{i_{1}}b_{i_{2}},p)=1$. Then take $I_1 \subset [3, \ell + \beta]$, with $|I_1| = \ell-2$ and $i_1, \,i_2 \in I_1$. According to Lemma \ref{luca} we can find $u_i \in A$, for all $i \in I_1$ such that
$$\sum_{i\in I_1} u_ib_i \equiv 0  \pmod {p^\beta}.$$
 Again by Lemma \ref{luca} we can now find $u_{1},u_{2}\in A$ such that

$$u_1a_1+u_2a_2+\sum_{i\in I_1} u_ia_i \equiv 0 \pmod {p^{\alpha}}.$$
Hence
$$u_1x_1 + u_2x_2 + \sum_{i\in I_1} u_ix_i = (0,0) \;\;\mbox{in $G$},$$
which gives an $A$-weighted zero-sum subsequence of $S$ of length $\ell$.

Now let us  suppose that $b_i \equiv 0 \pmod p$ for every index $i\in[3, \ell+ \beta-1]$, and write  $b_i = c_ip$, for all $i\in[3, \ell+ \beta-1]$. The sequence
$$S' = (1,0)(1,0)(a_3,c_3)(a_4,c_4)\cdots (a_{\ell+\beta-1},c_{\ell+\beta-1})$$
is now over $\Z_{p^\alpha}\oplus\Z_{p^{\beta-1}}$ and has length $\ell + \beta-1$. By the induction hypothesis, there is an $A$-weighted zero-sum subsequence of $S'$ of length $\ell$. Multiplying the second coordinate of terms of this subsequence by $p$, we obtain the required result for $S$.
\end{proof}
\begin{lemma} \label{fi10}
Let   $G = \Z_{p^{\alpha}}\oplus \Z_{p^\beta}$,  $\alpha \ge \beta$ and $\ell \ge 4$.  Let $S = x_1x_2\cdots x_{\ell + \beta}$ be a sequence over $G$, with $x_i = (a_i, b_i)$, for  $i \in [1, \ell + \beta]$. If for some $j\in [1, \ell+\beta]$ we have $\gcd(a_j,p)=1$ and there   are $i_1$, $i_2$, $\ldots$ , $i_k \in [, \ell + \beta]\setminus\{j\}$ with $k \leq \ell-3$ such that
$$u_{i_1}x_{i_1} + u_{i_2}x_{i_2} + \cdots + u_{i_k}x_{i_k} = x_j,$$
for some choice of $u, u_{i_1}, u_{i_2}, \ldots, u_{i_k} \in A$, then $S$ has an $A$-weighted zero-sum subsequence of length $\ell$.
\end{lemma}
\begin{proof}  Let us assume $x_{1}=(a_{1},b_{1})$ with $\gcd(a_1,p)=1$ and
$$u_{2}x_{2} + \cdots + u_{k+1}x_{k+1} = x_1,$$
for some choice of $u_{2},\ldots, u_{k+1}\in A$. Consider the sequence
$$S' = x_1(u_{2}x_{2} + \cdots + u_{k+1}x_{k+1})x_{k+2}\cdots x_{\ell + \beta}$$
of length $\ell + \beta - ( k-1)$. This sequence has  the first two terms equal to $x_{1}$, and since $\ell  - ( k-1)\geq 4$, we can apply Lemma \ref{fi1} and find an $A$-weighted zero-sum subsequence $S^{''}$ of $S' $,
$$
S^{''} = x_1(u_{2}x_{2} + \cdots + u_{k+1}x_{k+1})x_{i_{3}}\cdots x_{i_{m}}
$$
of length $m=\ell -(k-1)$. Now it is easy to see that 
$$
S^{*} = x_1x_{2} \cdots x_{k+1}x_{i_{3}}\cdots x_{i_{m}}
$$
is an  $A$-weighted zero-sum subsequence of length $\ell$ of $S$.
\end{proof}

\section{Main Result}

Let us start this section recalling that 
$$
A=\left\{x\in\mathbb{N}\; : \; 1 \le a \le p^{\alpha} \; \mbox{ and }\; \gcd(a, p) = 1\right\}.
$$ 
Chintamani and Paul in \cite{CHI,CH2} proved that, for any $\alpha \in \N$,
\begin{equation}\label{CP}
s_A(\Z_{p^{\alpha}}\oplus\Z_{p}) = p^{\alpha} + \alpha + 1\; \mbox{ and }\; s_A(\Z_{p^{\alpha}}\oplus\Z_{p^{2}}) = p^{\max(\alpha,2)} + \alpha + 2.
\end{equation}
Our goal is to generalize these results and prove that, for any $\alpha, \,\beta \in \N$, we have
$$
s_A(\Z_{p^{\alpha}}\oplus\Z_{p^\beta}) = p^{\max(\alpha,\beta)} + \alpha + \beta.
$$

We begin this proof by considering  $S = x_1x_2 \cdots x_{m}$,
 a sequence over $G= \Z_{p^{\alpha}}\oplus\Z_{p^\beta}$, of length $m=p^{\max(\alpha,\beta)} + \alpha + \beta$. Let us write $x_i = (a_i, b_i) \in G$, for all $i\in [1,m]$.
 
 We are going to proceed by a simultaneous induction over $\alpha$ and $\beta$. The particular cases are given in \eqref{CP}, so we will assume that $s_A(\Z_{p^{\delta}}\oplus\Z_{p^\gamma}) \leq p^{\max(\delta,\gamma)} + \delta + \gamma$, for any $\delta + \gamma < \alpha + \beta$. In particular our induction hypothesis tells us that, if $\ell\geq p^{\max(\delta,\gamma)} + \delta + \gamma$, then any sequence of length $\ell$ over $\Z_{p^{\delta}}\oplus\Z_{p^\gamma}$ has an A-weighted zero-sum  subsequence of length $p^{\max(\delta,\gamma)}$ (the exponent of this group).

With the considerations above, we will prove some lemmas.
 \begin{lemma}\label{seq}
 If there is an $I\subset[1,m]$,  such that $a_{i}\equiv 0\pmod{p}$, for all $i\in I$ and $|I|\geq p^{\max(\alpha-1,\beta)} + \alpha + \beta-1$,  then the sequence $S$ has an A-weighted zero-sum subsequence of length $p^{\max(\alpha -1, \beta) }$ in $G$. Similarly, if  there is an $I\subset[1,m]$,  such that $b_{j}\equiv 0\pmod{p}$, for all $j\in I$,  and  $|I|\geq p^{\max(\alpha,\beta-1)} + \alpha + \beta-1$,  then the sequence $S$ has an A-weighted zero-sum subsequence of length $p^{\max(\alpha, \beta-1)}$ in $G$.
 \end{lemma}
\begin{proof} It is sufficient to considered one of the cases. Let $k=\max(\alpha, \beta-1)$ and, with no loss in generality, suppose $b_i=b_i^{'}p$ for $i\in [1,r]$, with $r=|I|$. Thus, the sequence
$$S' = (a_1, b_1') (a_2, b'_2) (a_3, b_3') \cdots (a_{r}, b'_{r})$$
is over $\Z_{p^{\alpha}}\oplus\Z_{p^{\beta-1}}$.  
Since $r\geq p^{\max(\alpha,\beta-1)} + \alpha + \beta-1$, 
this sequence $S'$ must contain an $A$-weighted zero-sum subsequence of length $p^{k}$, that is, there exist $J\subset [1,r]$, with $|J|=p^{k}$, and $u_{1},\ldots, u_{p^{k}}\in A$ such that
$$
\sum_{i\in J} u_{i}(a_{i}, b_{i}') = 0 \;\;\mbox{in} \;\;\Z_{p^{\alpha}}\oplus\Z_{p^{\beta-1}}.
$$
Therefore 
$$\sum_{i\in J} u_{i}(a_{i}, b_{i}) = 0 \;\;\mbox{in} \;\;\Z_{p^{\alpha}}\oplus\Z_{p^{\beta}},$$
as desired.
\end{proof}
\begin{lemma}\label{seq2} Suppose $\alpha >\beta$. If there is an $I\subset[1,m]$, $|I|=m-1$,  such that $a_{i}\equiv 0\pmod{p}$, for all $i\in I$, then the sequence $S$ has an A-weighted zero-sum subsequence of length $p^{\alpha}$ in $G$.
\end{lemma}
\begin{proof}  By the induction hypothesis, we are assuming that (for $\alpha > \beta$)
$$
s_{A}(\Z_{p^{\alpha-1}}\oplus\Z_{p^{\beta}}) \leq p^{\alpha-1} + \alpha + \beta -1.
$$
 On the other hand
\begin{equation}\label{m-1}
m-1= p^{\alpha} + \alpha + \beta -1 = p.p^{\alpha-1} + \alpha + \beta -1,
\end{equation}
hence we can apply Lemma \ref{seq} to obtain an A-weighted zero-sum subsequence $T_{1}$ of length $p^{\alpha-1}$. If we exclude  this  subsequence $T_{1}$ of the sequence $S$, we will have a subsequence of length $m - p^{\alpha -1}$ with  still  enough terms  (see \eqref{m-1}) to apply  again Lemma \ref{seq} and find another  disjoint A-weighted zero-sum subsequence $T_{2}$ of length $p^{\alpha -1}$. Observe that  this process can be repeated $p$ times (see \eqref{seq2}), hence we have $p$ disjoint A-weighted zero-sum subsequences $T_{1},T_{2},\ldots, T_{p}$ of $S$, each of length $p^{\alpha -1}$. Therefore
$$
T_{1}T_{2}\cdots T_{p}
$$
is an A-weighted zero-sum subsequence of $S$ of length $p^{\alpha}$ in $G$.
\end{proof}
\begin{lemma}
If $\alpha = \beta$ then the sequence $S$ has an A-weighted zero-sum subsequence of length $p^{\beta}$ in $G$.
\end{lemma}
\begin{proof}   Since $\alpha = \beta$ then $\max(\alpha, \beta-1)= \max(\alpha-1, \beta)= \beta$. According to Lemma \ref{seq}, we may assume  that $\gcd(b_{1},p)=1$, and applying Lemma \ref{h4} 
(for $\alpha = \beta$), we may write, abusing notation,
$$
S = (0,1)(a_{2},b_{2})(a_{3},b_{3}) \cdots (a_{m},b_{m}).
$$
Again, according to Lemma \ref{seq} we assume that $\gcd(a_{2},p)=1$ and by Lemma \ref{h4}, we may rewrite $S$ as (also abusing notation)
$$S = (1, 0) (0, 1) (a_3, b_3) \cdots (a_{m}, b_{m}).$$
Now we can apply  Lemma \ref{seq} to the two  sequences $(0, 1)(a_3, b_3) \cdots (a_{m}, b_{m})$ and $(1, 0) (a_3, b_3) \cdots (a_{m}, b_{m}),$ and   guarantee that there must be an $a_{i}$ and an $b_{j}$, $i,j\in [3,m]$ (including the case where $i = j$), such that $\gcd(a_ib_j, p) = 1$.
We can now apply Lemma \ref{h2}, to conclude this proof.
\end{proof}
\begin{lemma}\label{a>b}
If $\alpha > \beta$ then the sequence $S$ has an A-weighted zero-sum subsequence of length $p^{\alpha}$ in $G$.
\end{lemma}
\begin{proof} According to Lemma \ref{seq2}, we  may assume that $\gcd(a_{1}a_{2},p)=1$. By Lemma \ref{h4} and abusing notation we may write $S$ as
$$
S = (1, 0) (a_2, b_2) (a_3, b_3) \cdots (a_{m}, b_{m}),
$$
and still have $\gcd(a_2,p)=1$.  From Lemma \ref{seq} we may assume that the $\gcd(b_{3}b_{t},p)=1$, for $t\in[2,m]$, $t\neq 3$. 

If  $a_{3}\equiv 0 \pmod{p}$, or if  $t\neq 2$ we have $a_{t}\equiv 0 \pmod{p}$, then we can apply Lemma \ref{h2} to obtain an $A$-weighted zero-sum subsequence of length $p^{\alpha}$, for $\gcd(b_{t},p)=1$. Thus let us assume (see Lemma \ref{ob1}) $S$ to be 
$$
S = (1, 0) (1, b_2) (1, b_3) \cdots (a_{m}, b_{m}),
$$
and $\gcd(b_{2}b_{3},p)=1$.
If for some $i \in[4, m]$  we have  $\gcd(a_i, p) = 1$ and $b_i\equiv 0 \pmod{p}$ then
$$(a_ib_3 - b_i)x_2 - a_ib_2x_3 + b_2x_i = (a_ib_3 - b_i)(1,0).$$
Since $a_ib_3 - b_i \in A$, there exists a $w\in\Z_{p^{\alpha}}$ such that
$(a_ib_3 - b_i)w=1$. Hence we can rewrite the expression above as
$$
x_{2} + u_{3}x_{3} + u_{i}x_{i} = (1,0)
$$
with $u_{3},u_{i}\in A$. Taking $k=3$ and $\ell = p^{\alpha}\geq 9$ , for p is an odd prime and $\alpha \geq 2$, it follows by the Lemma \ref{fi10}  that $S$ has an $A$-weighted zero-sum subsequence of length $p^{\alpha}$.

Hence, for some $t\in[3,m]$,   we may  consider $S$ to be  (by  Lemma \ref{ob1})
\begin{equation}\label{S}
S= (1,0)(1,b_2)(1,b_3)\cdots(1,b_t)(a_{t+1},b_{t+1})\cdots(a_{m},b_{m}), 
\end{equation}
with $\gcd(b_{2}b_{3}\cdots b_{t}, p)=1$ and $a_{i}\equiv b_{i} \equiv 0 \pmod{p}$, for all $i\in [t+1,m]$.

If $b_i\not\equiv b_j \pmod p$, for  $i,j\in\{2,3,\ldots, t\}$, then 
$$
b_ix_j - b_jx_i = (b_i-b_j)(1, 0).
$$ 
Since $b_i-b_j\in A$, we can repeat the arguments above, using Lemma \ref{fi10},
with $\ell = p^{\alpha}$ and $k=2$,  to obtain  an $A$-weighted zero-sum subsequence of length $p^{\alpha}$. So let us assume that $b_{j}=b_{2} + k_{j}p$, for all $j\in [3,m]$.

By the proof of Lemma \ref{h4}, we see that the automorphism $\Theta(x,y)= (x, y - b_2\mu(x))$ has the following properties:
$$
\Theta(1,b_{2})=(1,0),\; \Theta(1, 0) = (1,-b_{2}), \; \Theta(1, b_{2}+ np) = (1,np),\; \mbox {and} \;\Theta(a, bp) = (a, b'p).
$$

Thus  (see \eqref{S}),
$$\Theta(S)= (1,-b_2)(1,0)(1,k_3p)\cdots(1,k_tp)(a_{t+1},pb'_{t+1})\cdots(a_{m},pb'_{m}),$$
that is,  $\Theta(S)= (c_{1},d_{1})\cdots (c_m ,d_{m})$ and $d_{j}\equiv 0 \pmod{p}$ for all $j\in [2,m]$. As $m-1\geq p^{\alpha-1} + \alpha + \beta-1$, we can apply Lemma \ref{seq} to obtain an $A$-weighted zero-sum subsequence of $\Theta(S)$ of length $p^{\alpha}$, for we are assuming $\alpha > \beta.$ Now the result follows from Lemma \ref{ob1}.
\end{proof}

\begin{theorem}\label{princ}
Let $p$ be an odd prime and $G = \Z_{p^{\alpha}}\oplus\Z_{p^\beta}$, with  $\alpha \ge \beta$. Then $s_A(G) = p^{\alpha} + \alpha + \beta$.
\end{theorem}
 \begin{proof}
 The sequence of lemmas above proved that $s_A(G) \leq p^{\alpha} + \alpha + \beta$.  We conclude this prove presenting the following sequence over $\Z_{p^\alpha}\oplus\Z_{p^\beta}$ and length $p^\alpha+\alpha+\beta -1$, with no $A$-weighted zero-sum subsequence of length $p^\alpha$.
$$ \underbrace{(0,0) (0,0) \cdots (0,0)}_{p^{\alpha}-1\textrm{ terms}} (1, 0) (p, 0) (p^2, 0) \cdots (p^{\alpha-1}, 0) (0, 1) (0, p) \cdots (0,p^{\beta-1}). $$
This completes the proof of the theorem.
 \end{proof}
 
\section{Extremal A-weighted zero-sum free sequences}

Let $S$ be a sequence over $\Z_{p^{\alpha}}\oplus\Z_{p^\beta}$, and (see Theorem \ref{princ})
$$
m_{0}= p^{\max(\alpha, \beta)} + \alpha + \beta -1= s_{A}(\Z_{p^{\alpha}}\oplus\Z_{p^\beta})-1.
$$
 Let us denote by  $\delta_j(S)$ the number of terms (with multiplicity) of $S$ with order $p^j$, for all $j \in [1, \max(\alpha, \beta)]$. In the last section we have proved that $s_A(G) = p^{\max(\alpha,\beta)} + \alpha + \beta$, and presented a sequence  $S$ with no $A$-weighted zero-sum subsequence of length $p^{\max(\alpha,\beta)}$, with the following characteristics:  
\begin{equation}\label{carac}\begin{array}{l}
(1) \;\; S \;\;\mbox{contains}\;\; p^{\max(\alpha,\beta)}-1\;\;\mbox{ terms equal to } (0,0),  \\
(2) \;\; \delta_j(S)\geq 1\;\;\mbox{ for all } j \in [1, \max(\alpha,\beta)], \;\; \mbox {and }  \\
 (3) \;\;\sum_{j=1}^{\max(\alpha,\beta)} \delta_j(S) = \alpha + \beta.
\end{array}
\end{equation}

Here we want to prove all sequences with no $A$-weighted zero-sum subsequence of length $p^{\max(\alpha,\beta)}$ have the same characteristics.
 
 As before we are going to proceed by a simultaneous induction over $\alpha$ and $\beta$.  Chintamani and Paul proved in \cite{CHI, CH2} that this result is true whenever $\alpha \geq \beta$ and $\beta\in\{1,2\}$. Hence  we will assume that $\min(\alpha, \beta) \geq 3$ and  also assume that any sequence over $ \Z_{p^{\gamma}}\oplus\Z_{p^\delta}$, with $\delta + \gamma < \alpha + \beta$,  and of length $n=s_{A}(\Z_{p^{\gamma}}\oplus\Z_{p^\delta})-1$, with  no $A$-weighted zero-sum subsequence of length $p^{\max(\delta, \gamma )}$ has all the  characteristics described above, that is,
 \begin{equation}\label{carac2}
 \begin{array}{l}
(1) \;\; S \;\;\mbox{contains}\;\; p^{\max(\gamma,\delta)}-1\;\;\mbox{ terms equal to } (0,0),  \\
(2) \;\; \delta_j(S)\geq 1\;\;\mbox{ for all } j \in [1, \max(\gamma,\delta)], \;\; \mbox {and }  \\
 (3) \;\;\sum_{j=1}^{\max(\gamma,\delta)} \delta_j(S) = \gamma + \delta.
\end{array}
 \end{equation} 
 
 Let us start this proof by considering  $S = x_1x_2 \cdots x_{m_{0}}$,
 a sequence over $ \Z_{p^{\alpha}}\oplus\Z_{p^\beta}$. Let us  write $x_i = (a_i, b_i) \in G$, for all $i\in [1,m_{0}]$, and assume  that $S$ has no $A$-weighted zero-sum subsequence of length $p^{\max(\alpha, \beta )}$. 
 \begin{lemma}\label{5.1}
 Let us assume $\alpha \geq \beta$. Then the sequence $S$ must have at least one $x_i$ with $\gcd(b_i , p)=1$, and at least one $x_j$ with $\gcd(a_{j} , p)=1$
 \end{lemma}
 \begin{proof}
 If $b_{i}\equiv 0 \pmod{p}$ for $i \in [1,m_{0}]$ then, according to Lemma \ref{seq}, $S$ has an $A$-weighted  zero-sum subsequence of length $p^{\alpha}$, for we are assuming $\alpha \geq \beta$, that is, $m_{0}= p^{\alpha}+\alpha+\beta-1$, an absurd. Now suppose  $a_{i}\equiv 0 \pmod{p}$ for $i \in [1,m_{0}]$. If $\alpha = \beta$ just repeat the arguments above, changing the coordinates. If $\alpha > \beta$  then we can also  repeat the arguments but now using Lemma \ref{seq2} instead.
 \end{proof}
\begin{lemma}\label{b1}
 Let us assume $\alpha \geq \beta$. If $S$ has only one term $x_{j}$ such that
 $\gcd(b_{j},p)=1$ then $S$ has all the properties stated in \eqref{carac}.
\end{lemma}
\begin{proof}
With no loss in generality, let us assume $b_{j} \equiv 0\pmod{p}$ for all $j\in [2,m_{0}]$, and write $b_i=pb_i'$. Then the sequence
$$S' = (a_2, b_2') (a_3, b_3') (a_4, b_4') \cdots (a_{m_{0}}, b_{m_{0}}')$$
is over $\Z_{p^{\alpha}}\oplus\Z_{p^{\beta-1}}$ and has length $m_{0}-1$. Since $S$ has no  $A$-weighted zero-sum subsequences of length $p^\alpha$,  the same is true for $S'$. By the induction hypothesis, since 
$m_{0}-1 = s_{A}(\Z_{p^{\alpha}}\oplus\Z_{p^{\beta-1}}) -1$ , the sequence $S'$ has the properties stated in \eqref{carac2}, that is, $S'$ contains $p^{\alpha} -1 $ terms equal to $(0,0)$,  $\delta_{j}(S')\geq 1$, for all $j\in [1, \alpha]$ and  $\sum_{j=1}^{\max(\alpha)} \delta_j(S) = \alpha + \beta -1$.  But now observe that if $(a_j, b_j')$ has order $t$ in $\Z_{p^{\alpha}}\oplus\Z_{p^{\beta-1}}$, then $(a_j, b_j)=(a_j, pb_j')$ has also order $t$ in $\Z_{p^{\alpha}}\oplus\Z_{p^{\beta}}$. Now we can return to the sequence $S$, that has also $x_{1}=(a_{1},b_{1})$ with $\gcd(b_{1},p)=1$, that is, $x_{1}$ has order $p^{\alpha}$. Therefore $S$ has all the properties stated in \eqref{carac}.
\end{proof}
\begin{lemma}\label{a=b}
If $\alpha = \beta$ then the sequence $S$ has all the properties stated in \eqref{carac}. 
\end{lemma}
\begin{proof}
According to Lemma \ref{b1} we may assume $\gcd(a_{1}a_{2},p)=1$, and this can be done for $\alpha = \beta$. Applying Lemma \ref{h4}, we may consider $S$ to be (abusing notation)
$$
S = (1,0)(a_{2},b_{2})(a_{3},b_{3}) \cdots (a_{m_{0}},b_{m_{0}}).
$$
and we still have $\gcd(a_{2},p)=1$. Again according to \ref{b1}, we may assume
that $\gcd(b_{3}b_{t},p)=1$, for some $t\in [2,m_{0}]$, and $t\neq 3$. Again, by lemma \ref{h4} (for $\alpha = \beta$), we may rewrite $S$ as (also abusing notation)
$$S = (1, 0) (0, 1) (a_3, b_3) \cdots (a_{m_{0}}, b_{m_{0}}),$$
Now we can apply  Lemma \ref{b1} to the two  sequences $(0, 1)(a_3, b_3) \cdots (a_{m_{0}}, b_{m_{0}})$ and $(1, 0) (a_3, b_3) \cdots (a_{m_{0}}, b_{m_{0}})$, and   guarantee that there must be an $a_{i}$ and an $b_{j}$, $i,j\in [3,m]$ (including the case where $i = j$), such that $\gcd(a_ib_j, p) = 1$. 
But since $m_{0} - 1 > 3^{3}$, this gives a contradiction with Lemma \ref{h2}.
Hence this case does not occur, and the only possibility for $S$ is the one stated in Lemma \ref{b1}, that is, $S$ has only one term $x_{j}$ such that
 $\gcd(b_{j},p)=1$ and only one term $x_{i}$ such that
 $\gcd(a_{i},p)=1$, for $\alpha = \beta$.
\end{proof}
\begin{lemma}\label{a1}
 Let us assume $\alpha >\beta$. If $S$ has only one term $x_{j}$ such that
 $\gcd(a_{j},p)=1$ then $S$ has all the properties stated in \eqref{carac}.
\end{lemma}
\begin{proof}
Let us assume $\gcd(a_{1},p)=1$ (see Lemma \ref{5.1}) and $a_{j} \equiv 0\pmod{p}$ for all $j\in [2,m_{0}]$, and write $a_i=pa_i'$, then the sequence
$$S' = (a_2', b_2) (a_3', b_3) (a_4', b_4) \cdots (a_{m_{0}}, b'_{m_{0}})$$
is over $\Z_{p^{\alpha-1}}\oplus\Z_{p^{\beta}}$ and has length $m_{0}-1$.
Since 
$$
m_{0}-1= p^{\alpha} + \alpha + \beta - 2 = (p-1)p^{\alpha -1}  \alpha + \beta + (p^{\alpha -1} -2),
$$
and $s_{A}(\Z_{p^{\alpha-1}}\oplus\Z_{p^{\beta}})= p^{\alpha -1} + \alpha + \beta -1$ (see Theorem \ref{princ}), hence from the sequence $S'$ 
we extract $p-1$ disjoint $A$-weighted zero-sum subsequences, $T_1, T_2, \ldots, T_{p-1}$, each of length $p^{\alpha -1}$.

Let $S_{2}= S'(T_{1}T_{2}\cdots T_{p-1})^{-1}$,  the remaining subsequence of $S'$, once we extract  all the  disjoint subsequences $T_1, T_2, \ldots, T_{p-1}$. Thus $|S_{2}| = p^{\alpha -1} + (\alpha -1) + \beta -1$. Since $S'$ has no $A$-weighted zero-sum subsequence of length $p^{\alpha}$ (otherwise $S$ would have $A$-weighted zero-sum subsequence of length $p^{\alpha}$ contradicting the hypothesis),  $S_{2}$ can not have an $A$-weighted zero-sum subsequence of length $p^{\alpha -1}$, for this subsequence together with the subsequences  $T_1, T_2, \ldots, T_{p-1}$ would give an $A$-weighted zero-sum subsequence of length $p^{\alpha}$ for $S'$, an absurd.
Hence $S_{2}$ is a sequence over $\Z_{p^{\alpha-1}}\oplus\Z_{p^{\beta}}$, of length $p^{\alpha -1} + (\alpha -1) + \beta -1 = s_{A}(\Z_{p^{\alpha-1}}\oplus\Z_{p^{\beta}}) - 1$, with no $A$-weighted zero-sum subsequence of length $p^{\alpha}$. Thus we can apply the induction hypothesis and assume that $S_{2}$ has $p^{\alpha-1}-1$ terms equal to $(0,0)$ and  $(\alpha -1) + \beta$ terms different from $(0,0)$.

Take $x$ an element of some $T_{i}$, that is $x|T_{i}$, for some $i\in [1,p-1]$, and consider the sequence $xS_{2}$. Thus
$$
|xS_{2}|= p^{\alpha -1} + (\alpha -1) + \beta = s_{A}(\Z_{p^{\alpha-1}}\oplus\Z_{p^{\beta}}),
$$
hence  it has an $A$-weighted zero-sum subsequence $S_{3}$ of length $p^{\alpha -1}$ and we must have 
$x|S_{3}$. 
Now consider the sequence $U=S_{2}T_{i}(S_{3})^{-1}$ of length 
$$p^{\alpha -1} +\alpha + \beta -2 = s_{A}(\Z_{p^{\alpha-1}}\oplus\Z_{p^{\beta}}) -1.$$
If $U$ has an $A$-weighted zero-sum subsequence $T^{*}_{i}$ of length $p^{\alpha -1}$, then the sequence $T_{1}\cdots T^{*}_{i}\cdots T_{p-1}S_{3}$
is an $A$-weighted zero-sum subsequence of $S'$ of length $p^{\alpha}$, contrary to the hypothesis. Therefore we can apply the induction hypothesis and assume that 
$U$ contains $p^{\alpha -1} -1$ terms equal to $(0,0)$ and $\alpha +\beta -1$ terms different from $(0,0)$. Since $p^{\alpha -1} > \alpha + \beta$, for $p\geq 3$ and  $\alpha >\beta\geq 3$, we must assume that $T_{i}$ has at least one term equal to $(0,0)$.

Let us recall that the sequence $S_{2}$ has $p^{\alpha -1}-1$ terms equal to $(0,0)$,  the sequence $T_{i}$ has at least one term equal to $(0,0)$ and the two sequence are disjoint. Now consider the sequence 
$$
V= S_{2}T_{i}(\underbrace{(0,0)(0,0)\cdots (0,0)}_{p^{\alpha-1} \mbox{ vezes}})^{-1}.
$$
 By the same argument applied for the sequence $U$, we can also considerer that the sequence $V$ has no $A$-weighted zero-sum subsequence of length $p^{\alpha -1}$. Since $|V|= s_{A}(\Z_{p^{\alpha-1}}\oplus\Z_{p^{\beta}}) -1$, we can apply the induction hypothesis and assume that $V$ has $p^{\alpha -1} -1$ terms equal to $(0,0)$ and $\alpha + \beta -1$ terms different from $(0,0)$. Hence the sequence $S_{2}T_{i}$ has $2p^{\alpha -1} -1$ terms equal to $(0,0)$ and $\alpha + \beta -1$ terms different from $(0,0)$. But the sequence $S_{2}$ has exactly $p^{\alpha -1} -1$ terms equal to $(0,0)$ and $\alpha + \beta -1$ terms different from $(0,0)$, therefore we must have $T_{i} = (0,0)\cdots (0,0)$.
 
 Since this is true for any $T_{i}$, and $S_{2}$  have $p^{\alpha -1} -1$ terms equal to $(0,0)$, we have just proved that the sequence $S'$ has $p^{\alpha} -1$
 terms equal to $(0,0)$ and $\alpha + \beta -1$ terms different from $(0,0)$. Consequently, the sequence $S$ has $p^{\alpha} -1$
 terms equal to $(0,0)$ and $\alpha + \beta$ terms different from $(0,0)$, considering the term $x_{1}$ of order $p^{\alpha}$. This completes this proof.
\end{proof}

\begin{theorem} \label{princ2} Let $p$ be an odd prime and $S$ a sequence over the group $\Z_{p^\alpha}\oplus\Z_{p^\beta}$ with $\alpha \ge \beta$, and $|S| = p^\alpha + \alpha + \beta -1$. If $S$ has no $A$-weighted zero-sum subsequence of length $p^{\alpha}$, then $S$ has all the properties stated in \eqref{carac}.
\end{theorem}

\begin{proof} Let us consider as before  $S = x_1x_2 \cdots x_{m_{0}}$,
with  $x_i = (a_i, b_i)$, for all $i\in [1,m_{0}]$, and assume that $S$ has no $A$-weighted zero-sum subsequence of length $p^{\alpha}$. According to Lemmas  \ref{5.1}, \ref{a=b} and \ref{a1}, we may also assume that $\alpha > \beta$ and, with no loss in generality, that $\gcd(a_{1}a_{2},p)=1$.

Applying Lemma \ref{h4}, we can rewrite $S$ as 
$$S = (1, 0) (a_2, b_2) (a_3, b_3) \cdots (a_{m_{0}}, b_{m_{0}}),$$
and still have $\gcd(a_{2},p)=1$. By Lemmas \ref{5.1} and \ref{b1} we may assume, without loss of generality, that $\gcd(b_{3}b_{t},p)=1$, for some $t\in [2, m_{0}]$ and $t\neq 3$. As $\gcd(a_{2}b_{3},p)=1$, if there is $i\in[4,p^\alpha + \alpha + \beta-1]$ such that $\gcd(b_i, p) = 1$ and $p|a_i$, then we can apply Lemma \ref{h2} and find an $A$-weighted zero-sum subsequence of $S$ of length $p^{\alpha}$, a contradiction. 

Hence let us assume, with no loss in generality, and applying Lemma \ref{ob1}, that $S$ can be rewritten as
$$S = (1, 0) (1, b_2) (1, b_3) \cdots (a_{m_{0}}, b_{m_{0}}),$$
with $\gcd(b_{2}b_{3},p)=1$. If there is $i \in[4, m_{0}]$ such that $\gcd(a_i, p) = 1$ and $p|b_i$, then we have
$$(a_ib_3 - b_i)x_2 - a_ib_2x_3 + b_2x_i = (a_ib_3 - b_i)(1,0),$$
with $(a_ib_3 - b_i)\in A$. As done before, this shows that the hypothesis of Lemma \ref{fi10} are satisfied, and consequently $S$ has an $A$-weighted zero-sum subsequence of length $p^\alpha$, giving also a contradiction.
The only situation left to be analysed is to consider $S$ as (see Lemma \ref{ob1})
\begin{equation}\label{S0}
S= (1,0)(1,b_2)(1,b_3)\cdots(1,b_t)(a_{t+1},b_{t+1})\cdots(a_{m_{0}},b_{m_{0}}),
\end{equation}
for some  $3 \le t \le m_{0}$, and $\gcd(b_2b_3\cdots b_t,p)=1$, $p|a_i$ and $p|b_i$, for all $i\ge t+1$.
We are now in the same situation presented in the proof of Lemma \ref{a>b},
and we repeat the arguments for completion. 
If $b_i\not\equiv b_j \pmod p$, for  $i,j\in\{2,3,\ldots, t\}$, then 
$$
b_ix_j - b_jx_i = (b_i-b_j)(1, 0).
$$ 
Since $b_i-b_j\in A$, we can repeat the arguments above, using Lemma \ref{fi10},
with $\ell = p^{\alpha}$ and $k=2$,  to obtain  an $A$-weighted zero-sum subsequence of length $p^{\alpha}$, which gives a contradiction. So let us assume that $b_{j}=b_{2} + k_{j}p$, for all $j\in [2,m]$.

By the proof of Lemma \ref{h4}, we see that the automorphism $\Theta(x,y)= (x, y - b_2\mu(x))$ has the following properties:
$$
\Theta(1,b_{2})=(1,0),\; \Theta(1, 0) = (1,-b_{2}), \; \Theta(1, b_{2}+ np) = (1,np),\; \mbox {and} \;\Theta(a, bp) = (a, b'p).
$$

Thus  (see \eqref{S0}),
$$\Theta(S)= (1,-b_2)(1,0)(1,k_3p)\cdots(1,k_tp)(a_{t+1},pb'_{t+1})\cdots(a_{m_{0}},pb'_{m_{0}}),$$
that is,  $\Theta(S)= (c_{1},d_{1})\cdots (c_{m_{0}} ,d_{m_{0}})$ and $d_{j}\equiv 0 \pmod{p}$ for all $j\in [2,m]$. Now we can use Lemma \ref{b1} to conclude that this sequence satisfies the properties stated in \eqref{carac}. Since the any automorphism  preserves the order of the elements of a group, we can apply the inverse automorphism $\Theta^{-1}$ and conclude that  $S$ also  satisfies the properties stated in \eqref{carac}, completing the proof of the theorem.
\end{proof}

\section{The constant $\eta_{A}(\Z_{p^\alpha} \oplus \Z_{p^\beta})$}

For an odd prime $p$, Theorems \ref{princ} and \ref{princ2} present the exact value of $s_A(\Z_{p^\alpha} \oplus \Z_{p^\beta})$  and also classify all the sequences of length $s_A(\Z_{p^\alpha}\oplus\Z_{p^\beta}) - 1$ which are $A$-weighted zero-sum free sequences of length $p^\alpha$. A direct consequence of these results is the following theorem.

\begin{theorem} Let $p$ be an odd prime and $G = \Z_{p^{\alpha}}\oplus\Z_{p^\beta}$.  Then $\eta_A(G) =  \alpha + \beta + 1$.
\end{theorem}
\begin{proof}
Let $S$ be a sequence over $G$ of length $\alpha + \beta + 1$ and consider the sequence
$$
S^{*} = S\underbrace{(0,0) (0,0) \cdots (0,0)}_{p^{\alpha}-1\textrm{ terms}}.
$$
By theorem \ref{princ}, this sequence has an $A$-weighted zero-sum  subsequence of length $p^\alpha$. Since $p^{\alpha} \geq \alpha + \beta + 1$, then $S$ must then have an $A$-weighted zero-sum  subsequence of length $n\leq p^{\alpha}$ as desired.

On the other hand it is simple to see that the sequence of length $ \alpha + \beta$
$$
(0,1)(0,p)\cdots (0, p^{\beta -1})(1,0)(p,0)\cdots (p^{\alpha -1},0)
$$
has no $A$-weighted zero-sum  subsequence of length $n\leq p^{\alpha}$.
\end{proof}

Now the following corollary  is immediate.
\begin{corollary} Let $p$ be an odd prime and $G = \Z_{p^{\alpha}}\oplus\Z_{p^\beta}$.  Then $s_{A}(G)= \eta_A(G) + \mathrm{exp}(G) - 1$.
\end{corollary}
\begin{theorem}Let $p$ be an odd prime and $S$ a sequence over the group $\Z_{p^\alpha}\oplus\Z_{p^\beta}$ with $\alpha \ge \beta$, and $|S| =  \alpha + \beta$. If $S$ has no $A$-weighted zero-sum subsequence of length $ n\leq p^{\alpha}$, then  $\delta_j(S) \ge 1$, for all $j \in [1, \alpha]$, and $\sum_{j=1}^{\alpha} \delta_j(S) = \alpha + \beta$.
\end{theorem}
\begin{proof} Consider again the sequence 
$$
S^{*} = S\underbrace{(0,0) (0,0) \cdots (0,0)}_{p^{\alpha}-1\textrm{ terms}},
$$
and observe that $S^{*}$ has no $A$-weighted zero-sum subsequence of length $p^{\alpha}$.  By Theorem \ref{princ2}, the sequence $S^{*}$ contains $p^{\alpha} -1$ terms equal to $(0,0)$, $\delta_j(S^{*}) \ge 1$, for all $j \in [1, \alpha]$, and $\sum_{j=1}^{\alpha} \delta_j(S^{*}) = \alpha + \beta$. Hence the sequence $S$ has the desired properties stated in the theorem.
\end{proof}


\end{document}